\documentclass[12pt, reqno]{amsart}
\usepackage{amsmath, amsthm, amscd, amsfonts, amssymb, graphicx, color}
\usepackage[bookmarksnumbered, colorlinks, plainpages]{hyperref}
\usepackage{float}
\usepackage{lineno}
\usepackage{graphicx,epstopdf}
\usepackage{bbm}
\usepackage{mathrsfs}
\usepackage{dsfont}
\usepackage{amssymb}
\usepackage{inputenc}
\usepackage{csquotes}
\usepackage[dvips]{epsfig}
\usepackage{latexsym}
\usepackage{exscale}
\newtheorem{theorem}{Theorem}[section]
\newtheorem{lemma}[theorem]{Lemma}

\newtheorem{corollary}[theorem]{Corollary}
\theoremstyle{definition}
\newtheorem{definition}[theorem]{Definition}
\newtheorem{example}[theorem]{Example}

\theoremstyle{remark}
\newtheorem{remark}[theorem]{Remark}
\numberwithin{equation}{section}

\begin{document}

\setcounter{page}{1}

%\title[Property(E), Reflexivity and Bernstein Lethargy Theorem ]{ Property (E), Reflexivity and Bernstein Lethargy Theorem }
\title{On a theorem of Nikol'skii}

\author[A. AKSOY, Q. Peng]{Asuman G\"{u}ven AKSOY$^*, $ Qidi Peng$^1$}

\address{$^{*}$Department of Mathematics, Claremont McKenna College, 850 Columbia Avenue, Claremont, California  91711, USA.}
\email{\textcolor[rgb]{0.00,0.00,0.84}{aaksoy@cmc.edu}}
\address{$^{1}$Institute of Mathematical Sciences, Claremont Graduate University, 1237 N. Dartmouth Avenue, Claremont, California  91711, USA.}
\email{\textcolor[rgb]{0.00,0.00,0.84}{qidi.peng@cgu.edu}}

%\dedicatory{This paper is dedicated to Professor ABCD}

\subjclass[2010]{Primary 41A25; Secondary 41A50, 46B20.}

\keywords{Best approximation, Bernstein's lethargy theorem, reflexive Banach space.}

%\date{Received: xxxxxx; Revised: yyyyyy; Accepted: zzzzzz.
%\newline \indent $^{*}$Corresponding author}

\begin{abstract}
We present Bernstein lethargy theorem  and  examine the relationship between  Bernstein lethargy theorem and reflexivity.
%%%%%%%%%%%%%%
\iffalse

The equivalence of existence of best approximation  and the reflexivity is well known, moreover equivalence of reflexivity and Bernstein lethargy theorem was proved for a finite sequence of nonnegative numbers with the property 
$e_1\geq e_2\geq \dots\geq e_N$ and $e_{N+1}=0$ for $n\geq N$. In this paper, we consider existence of best approximation together with a sequence of nonnegative numbers  $\{d_n\}$ satisfying

% \label{condition_Borodin}
 $$ d_n >  \sum_{k = n+1}^\infty d_k \quad \mbox{for all}\quad  n \ge n_0 \quad \mbox{ at which}\quad  d_n > 0 $$
  %\end{equation}
  and investigate the relationship with Benstein lethargy theorem and reflexivity.
\fi
%%%%%%%%%%%%
\end{abstract} \maketitle

\section{Introduction}
%%%%%%%%%%%%
%%%%%%%%%%%%%%%%%
%%%%%%%%%%%%%%%%%%%%
%%%%%%%%%%%%%%%%%%%%%
\subsection{Bernstein Lethargy Theorem}
 One of the notable theorems used in the constructive theory of functions, called the Bernstein's Lethargy Theorem (BLT) \cite{Bernstein}. For a subspace $Y$ of a normed linear space $(X,\|\cdot\|)$,  define the distance from an element $x \in X$ to $Y$ by
$$\rho(x, Y):= \inf \{ \|x-y\|: \,\,y\in Y\}.$$
The Weierstrass Approximation Theorem states that polynomials are dense in $C[0,1]$, thus it is known that for any $f\in C[0,1]$, $$\displaystyle \lim_{n \to \infty} \rho(f, \Pi_n) = 0,$$
where for each $n\ge1$,  $\Pi_n$ is the vector space of all real polynomials of degree at most equal to $n$.
However, the Weierstrass Approximation Theorem gives no information about the speed of convergence for $\rho(f, \Pi_n)$. In $1938$, S.N. Bernstein \cite{Bernstein} proved that if $\{d_n \}_{n\ge1}$ is a non-increasing  sequence  of positive numbers with $\lim\limits_{n\to\infty}d_n=0$, then there exists a function $f\in C[0,1]$ such that
$$\rho(f, \Pi_n)= d_n,~\mbox{for all $n \ge1$}.$$
This remarkable result is called the Bernstein's Lethargy Theorem (BLT) and is used in the constructive theory of functions \cite{Sin}, and  in the theory of quasi-analytic functions in several complex variables \cite{Ple2}.

More generally, let $\{Y_n\}$ be a system of strictly nested subspaces of the Banach space $X$. The sequence of errors of the best approximation from $x\in X$ to $Y_n$, denoted by $\{\rho(x,Y_n)\}$, may converge to zero at an arbitrarily slow rate  or an arbitrarily fast rate. For example, Shapiro in \cite{Sha}, replacing $C[0,1]$ with an arbitrary Banach space $(X, \|\cdot\|)$ and $\{\Pi_n\}$, with a sequence $\{Y_n\}$ where $Y_1 \subset Y_2 \subset \cdots$ are strictly embedded closed subspaces of $X$, showed that in this setting, $\rho(x, Y_n)$ can decay arbitrarily slowly. More precisely, for each null sequence (i.e. a sequence converging to $0$) $\{d_n\}$ of non-negative numbers, there exists $x\in X$ such that
\begin{equation*}
\label{Shap}
 \rho(x, Y_n) \neq \mathcal O (d_n),~\mbox{as $n\to\infty$}.
 \end{equation*}
 This result is further improved  by Tyuriemskih \cite{Tyu}, who showed that for any expanding sequence $\{Y_n\}$ of subspaces of $X$ and for any null sequence $\{d_n\}$ of positive numbers, there is an element $x\in X$  such that
 \begin{equation*}
 \label{tyu}
 \lim_{n \rightarrow \infty} \rho(x, Y_n) =0,~\mbox{and}~\rho(x, Y_n) \geq d_n~\mbox{for all $n\ge1$}.
   \end{equation*}
   However, it is also possible that the errors of the best approximation $\{\rho(x,Y_n)\}$ may converge to zero arbitrarily fast. For results of this type see   \cite{Al-To}.

    We refer the reader to \cite{Ak-Al,Ak-Le,Al-To, Borodin1,Lew1} for other versions of the Bernstein's Lethargy Theorem  and  to \cite{Ak-Le2,Alb,Mich,Ple1,Vas} for the Bernstein's Lethargy Theorem for Fr\'{e}chet  spaces. Consult \cite{De-Hu} for an application of Tyuriemskih's Theorem to convergence of sequence of bounded linear operators and  \cite{Al-Oi} for a generalization of Shapiro's Theorem.
    %%%%%%%%%%%%
    We first  state a  Bernstein Lethargy result concerning \textit{finite number} of subspaces, for the proof  of the following lemma we refer the reader to Timan's book  \cite{Tim}.
\begin{lemma}
\label{lemma1}
Let $(X,\|\cdot\|)$ be a normed linear space, $Y_1\subset Y_2\subset\ldots\subset Y_n\subset X$ be a finite system of strictly nested subspaces, $d_1>d_2>\ldots>d_n\ge0$ and $z \in X\backslash Y_n$. Then, there is an element $x\in X$ for which $\rho(x,Y_k)=d_k$ $(k=1,\ldots,n)$,  $\|x\|\le d_1+1$, and $x-\lambda z\in Y_n$ for some $\lambda>0$. If all $Y_k$ are existence subspaces (in particular, if all $Y_k$ are finite-dimensional or $X$ is reflexive),
then the above assertion holds for arbitrary numbers $d_1 \ge d_2 \ge \ldots\ge d_n \ge 0$.
\end{lemma}

    %%%%%%%%%%%

 The general question is that, given a Banach space $X$ and a sequence of strictly increasing subspaces $\{Y_n\}$ of $X$  with a dense union and a decreasing null sequence $ \{d_n\}$, does there exist an $x\in X$ such that $\rho(x, Y_n) =d_n$? If such an element exists for all such subspaces $Y_n$ and numbers $d_n$ , then we say $X$ has the $(B)$-property. Besides Hilbert spaces \cite{Tyu2},  there are no known examples of infinite-dimensional  spaces $X$ having the $(B)$-property.
  Although for a long time no sequence $\{d_n\}$ of this type was known for which such an element $x$ exists for \textit{all} possible Banach spaces $X$, Borodin in \cite{Borodin1} showed that the problem has a solution whenever the sequence $\{d_n\}$ satisfies some fast convergence condition, or the subspaces $\{Y_n\}$ satisfy some  conditions. The first result in \cite{Borodin1}, obtained from the sequence condition (\ref{condition_Borodin}) is stated below:
\begin{theorem}{\cite[p. 624]{Borodin1}}
  \label{thm:Borodin1}
  Let $X$ be an arbitrary infinite-dimensional Banach space, $Y_1 \subset Y_2
  \subset \cdots$ be an arbitrary countable system of strictly nested subspaces
  in $X$, and fix a numerical sequence $\{d_n\}_{n\ge1}$ for which there exists a natural number $n_0\ge1$ such that
  \begin{equation}
  \label{condition_Borodin}
  d_n >
  \sum_{k = n+1}^\infty d_k~\mbox{for all $n \ge n_0$ at which $d_n > 0$.}
  \end{equation}
  Then there is an element $x \in X$ such that
  \begin{equation}
  \label{rhod}
  \rho(x,Y_n) = d_n, ~\mbox{for all $n\ge1$}.
  \end{equation}
\end{theorem}
The condition (\ref{condition_Borodin}) on the sequence $\{d_n\}$ is the key to the derivation of (\ref{rhod}) in Theorem \ref{thm:Borodin1}.  Note that, the condition (\ref{condition_Borodin}) is satisfied when $d_n =(2+\epsilon)^{-n} $ for $\epsilon > 0$ arbitrarily small, however it is not satisfied when $d_n= 2^{-n}$. 
The second result in \cite{Borodin1}, based on the subspace condition (\ref{Borodin2}) is given below:
\begin{theorem}{\cite[p. 626]{Borodin1}}
\label{thm:Borodin2}
	Let $d_0\ge d_1 \geq d_2 \geq \dots > 0$ be a non-increasing sequence converging to $0$ and $Y_1 \subset Y_2 \subset \dots \subset X$ be a system of strictly nested subspaces of  an infinite-dimensional Banach space $X$ for which there exists a series of non-zero elements $\{q_n\}$ such that $q_n \in Y_{n + 1} \setminus Y_n$, and
\begin{equation}
	\label{Borodin2}
\|q\| \leq \displaystyle \frac{d_{k-1}}{d_{k}}\rho(q,Y_k)
\end{equation}
holds for all $k \in \mathbb N$ and any non-zero element $q$ in the linear span $\langle q_k, q_{k + 1},\dots \rangle$. Then there is some element $x$ in the closed linear span $\overline{\langle q_1, q_2,\dots \rangle}$ satisfying $$\rho(x,Y_n) = d_n ~~for ~all ~n \ge1.$$
\end{theorem}
     The latter result is improved by Aksoy et al. in  \cite{Aksoy}. Note that Borodin's proof works only under the assumption that $\overline{Y}_n$ is strictly included  in $Y_{n+1}$ . Necessity of this assumption on subspaces is illustrated by the following example.
    %%%%%%%%%%%%%
  \begin{example}\label{EXAM}
Let $X = L^{\infty}[0,1]$ and consider $C[0,1] \subset L^{\infty} [0,1]$ and define the subspaces  of $X$ as follows:

\begin{enumerate}
\item $Y_1 =  \mbox{space of all polynomials}$;

\item $Y_{n+1} = $span$[Y_n \cup \{f_{n}\}]$
where $f_{n} \in C[0,1] \setminus Y_n$, for $n\ge1$.
\end{enumerate}
Observe that by the Weierstrass Theorem we have $\overline {Y}_n = C[0,1] $ for all $n \ge1$.
Take any  $ f \in L^{\infty}[0,1]$ and consider the following cases:
\begin{enumerate}
\renewcommand{\labelenumi}{\alph{enumi})}
\item If  $ f \in C[0,1]$ , then
$$
\rho (f,Y_n) = \rho(f,C[0,1]) = 0~\mbox{for all $n\ge1$}.
$$
\item If $ f \in L^{\infty}[0,1] \setminus C[0,1] $, then
$$
\rho(f, Y_n) = \rho(f, C[0,1]) = d > 0~ \mbox{ (independent of $n$)}.
$$
Note that in the above, we have used the fact that $\rho(f,Y_n) = \rho(f,\overline{Y}_n)$.
Hence in this case Bernstein lethargy theorem does not hold.
\end{enumerate}
\end{example}

    %%%%%%%%%%%%%
      It is natural to ask the relationship between reflexivity and Bernstein Lethargy Theorem. After all, reflexive spaces have better approximation properties, such as the existence of bounded projections or the weak compactness of closed balls. Moreover, certain versions of Bernstein Lethargy Theorem require the space itself or subspaces to be reflexive. We examine this in more detail in the next section.

%%%%%%%%%%%%%%
%%%%%%%%%%%%%
%%%%%%%%%%%%
 %%%%%%%%%%%%%%
\subsection{Existence Property and Reflexivity}

Let $X$ be a real or complex Banach space and $Y$ be a closed linear subspace of $X$. An element $y_0\in Y$ is called the \textit{the best approximation or the least deviation} from $x\in X$ if $$||x-y_0||=\inf_{y\in Y} ||x-y||=\rho(x, Y).$$ 
We say a Banach space $X$ has the existence property or \textit{Property (E)} for short, if for every choice of closed subspace $Y$ and the vector $x\in X$ there is at least one best approximation $y_0$  for $x$ in $Y$. 
\begin{theorem}\cite{Hir}\label{Hir}
Let $X$ be a Banach space, then $X$ has property $(E)$ if and only if $X$ is reflexive.
\end{theorem}
If $X$ is reflexive then any closed linear subspace $Y$ of $X$ is reflexive and from the fact that  in a reflexive space, the  closed unit ball is weakly compact, shows that $X$ has the property $(E)$.  The other implication comes from the result of R.C. James in  \cite{RCJ}, where he proves that if every linear functional assumes its maximum on absolute value on the unit sphere of any closed subspace of $X$, then $X$ is reflexive.
Another  equivalent way of looking at reflexive spaces are through the concept of proximal subspaces.  We call  $Y_n$ \textit{proximinal} in $X$ if for any $x\in X$ there exists $v\in Y_n$ such that
 $$\rho(x, Y_n) =\|x-v\|.$$ By letting $y=x-v$ we obtain that if    $Y_n$ is  proximinal in  $Y_{n+1}$, then given $x\in Y_{n+1}$,  there is $v\in Y_n$ such that $\rho(x,Y_n)=\|x-v\|$ , this condition reduces to for each $n\ge1$, there exists $y\in Y_{n+1}\backslash Y_n$ such that
\begin{equation}
\label{Yn'}
\rho(y, Y_n)=\|y\|
\end{equation}

 Thanks to the characterization of reflexivity in terms of the proximinality of its closed subspaces, the condition (\ref{Yn'})  holds for reflexive Banach spaces (see \cite[Theorem $2.14$, p. $19$]{Singer2}). \\
 
 Next we list some properties of the distance function $\rho(x, Y_n)$:\\[.02in]
Given a Banach space $X$ and its subspaces $Y_1\subset Y_2 \subset \dots \subset Y_n\subset \cdots$, it is clear that $$ \rho(x, Y_1) \geq\rho(x,Y_2)\geq \cdots, ~\mbox{for any $x\in X$}$$ and  thus $\{\rho(x,Y_n)\}_{n\ge1}$ forms a non-increasing sequence of errors of best approximation from $x$ to $Y_n$, $n\ge1$.  Furthermore,  it is easy to show that
\begin{eqnarray*}
&&\rho(\lambda x, Y_n)= |\lambda| \rho(x, Y_n),~\mbox{for any}~ x\in X~\mbox{and}~\lambda\in\mathbb R;\\
&& \rho(x+v, Y_n)= \rho(x, \overline Y_n),~\mbox{for any}~x\in X~ \mbox{and}~v\in \overline Y_n;\\
&&\rho(x_1+x_2, Y_n)\leq \rho(x_1, Y_n)+\rho(x_2, Y_n)
\end{eqnarray*}
and consequently   $$\rho(x_1+x_2, Y_n)\geq |\rho(x_1, Y_n)-\rho(x_2, Y_n)|,~ \mbox{ for any}~  x_1,x_2\in X.$$
 Note that we also have: $$ |\rho(x_1, Y_n)-\rho(x_2, Y_n)| \leq \|x_1-x_2\|,~ \mbox{ for any}~  x_1,x_2 \in X,$$ which  implies that  the mapping $X \longrightarrow \mathbb{R}^+$  defined by $ x\longmapsto \rho(x, Y_n)$ is  continuous and thus properties of continuous mappings such as the intermediate value theorem can be used. Moreover 
 $ \rho(x+ty, Y_n)$ is a continuous function of $t$ and
 if $y \notin Y_n$ then $ \rho(x+ty, Y_n) \to \infty$ as $|t| \to \infty$.

\subsection{Bernstein Lethargy Theorem and Reflexivity}

\begin{definition}\cite{Nik}
A Banach space $X$  has the property $(B_f)$ if for every countable sequence of closed  distinct subspaces 
$$\{0\} = Y_0\subset Y_1\subset \dots\subset Y_n\subset \dots\subset X$$ and any finite sequence of nonnegative numbers
$$ e_0\geq e_1\geq \dots\geq e_N$$ there exists $x\in X$ for which $\rho(x,Y_n)=e_n$ for $n=1,2,\dots, N$ and $\rho(x,Y_n)=0$ when $n=N+1,\dots$. 

\end{definition}
\begin{theorem}\cite{Nik}
\label{B_f}
A Banach space  $X$ has the property $(B_f)$ if and only if $X$ is reflexive. 
\end{theorem}
Since the paper   \cite{Nik}  of Nikol'skii is  not available  in translation and appeared in a hard to reach regional  journal, in the following we give the proof by clarifying the steps.
\begin{proof}
Suppose $X$ is not reflexive,  there exists $Y$, non-reflexive subspace of $X$, such that its complementary $Y^c:={\rm span}(X\setminus Y)$ is infinite dimensional, and we use this to construct a sequence of closed subspaces $Y_0=\{0\}\subset \cdots \subset Y_n \subset Y_{n+1}\subset \cdots\subset X$ with certain properties. In particular, there is $y_0\in X$ such that $\rho(y_0,Y_1)=\rho_0>0$ but there are no best approximations to $y_0$ with the elements of $Y_1$. In particular, $\|y_0\|_X>\rho_0$ since otherwise $0$ would be a best approximation to $y_0$ from $Y_1$.  Moreover, $Y_2=\overline{{\rm span}\{y_0,Y_1\}}^X$. You claim that there is no $x\in X$ such that 
\[
\|x\|=\rho(x,Y_0)=\rho (x,Y_1)=\rho_0>0=\rho(x,Y_2)=\rho(x,Y_3)=\cdots 
\]
Clearly, $0=\rho(x,Y_2)$ means that $x\in Y_2=\overline{{\rm span}\{y_0,Y_1\}}^X$. If $x=\lambda y_0+y_1$ with $y_1\in Y_1$ and $\lambda\in \mathbb{R}$ , then $\rho(x,Y_1)=\rho(\lambda y_0,Y_1)=|\lambda|\rho_0=\rho_0$ implies $|\lambda|=1$, so $x=\pm y_0+y_1$. Then $0=\|x\|=\rho(x,Y_0)$ would imply $$\|\pm y_0+y_1\|=\|y_0\pm y_1\|=\rho_0=\rho(y_0,Y_1)$$ and $y_0$ would admit a best approximation from $Y_1$, which is impossible. 

Now, it may happen that $x=\lim_{n\to\infty}(\lambda_ny_0+y_n)$ for certain sequences $(\lambda_n)\subseteq \mathbb{R}$ and $(y_n)\subseteq Y_1$, but $x$ does not belong to ${\rm span}\{y_0,Y_1\}$.  In such a case, we would get
$$\rho(x,Y_1)=\lim_{n\to\infty}\rho
(\lambda_n y_0,Y_1)=\lim_{n\to\infty}|\lambda_n|\rho_0=\rho_0,$$ which just means $\lim_{n\to\infty}|\lambda_n|=1$. 

%We need to get a contradiction from these facts and the fact that 
%\begin{equation} \label{uno}
%    \rho_0=E(x,Y_0)=\|x\|=\lim_{n\to\infty} \|\lambda_ny_0+y_n\|_X. 
%\end{equation}

Now, $\lim_{n\to\infty}|\lambda_n|=1$ implies that we can extract a subsequence $(\lambda_{n_k})$ from $(\lambda_n)$ that converges either to $1$ or to $-1$ (this is so because we only deal with real numbers). Assume that $\lim_{k\to\infty}\lambda_{n_k}=1$. Then $x=\lim_{k\to\infty}(\lambda_{n_k}y_0+y_{n_k})= y_0+\lim_{k\to\infty}y_{n_k}$ means that $\lim_{k\to\infty}y_{n_k}=x-y_0$ exists, so it belongs to $Y_1$ (which is closed). A contradiction since we assumed that $x$ does not belong to ${\rm span}\{y_0,Y_1\}$! An analogous argument applies if $\lim_{k\to\infty}\lambda_{n_k}=-1$. 

%%%%%%%%%%%%%%%%%%%
\iffalse
then there exists a non-reflexive subspace $Y$  such that $X\setminus Y$ is infinite dimensional.  Based on the Theorem \ref{Hir} there is an element $y_0\in Y$ and a closed subspace $Y^* \subset Y$ such that for $y_0$ there does not exist in $Y^*$ a least deviation.  Suppose $\rho(y_0, Y^*)= \rho_0> 0$. Next set $Y_1=Y^*$ and let $Y_2 $ be the closed linear span of $y_0$ and $Y^*$. In this manner we can define a countable sequence $\{Y_n\}_{n=3}^{\infty}$ where all $Y_n$'s are distinct closed subspaces of $X$. This is possible since $X\setminus Y$ is infinite dimensional.  Now, define numbers $e_n$ with
$$ e_0=e_1=\rho_0, \quad \quad e_n=0\quad\mbox{for}\quad n\geq 2.$$  $y_0$ is the unique element up to constant $\epsilon$, $|\epsilon|=1$ that satisfies the following conditions:
$$ y_0\in Y_2 \quad \mbox{and} \quad  \rho(y_0, Y_1)=e_1=\rho_0$$
The equation $\rho( y_0, Y_0)=  e_0= \rho_0$ can not hold. Since in this case there is an element  in $Y_1$  that will be the least deviation from $y_0$. Thus there does not exist $Y_1=Y^*$ with the $\rho(y_0, Y_1)=e_1=\rho_0$.  
\fi
%%%%%%%%%%%%%%%%%%%%%%%%%%
i.e., if $X$ is not reflexive  then $X$ does not have the property $(B_f)$.

To prove the other implication, suppose $X$ is reflexive, take any element $y_{N+1} \in Y_{N+1} \setminus Y_N$, then find $t=c_{N+1}$ such that $$ \rho( c_{N+1} y_{N+1}, Y_N)= e_N.$$ Since $X$ is reflexive, there exists an element $y\in Y_N$ which is the least deviation from $ c_{N+1} y_{N+1}$.

Hence we have $$ \rho( c_{N+1} y_{N+1}+ y , Y_0)= e_N.$$

Let $y\in Y_m\setminus Y_{m-1}$, consider the following possibilities:
%%%%%%%%%%%%%%%%
\begin{enumerate}
\renewcommand{\labelenumi}{\alph{enumi})}

%\begin{description}
\item$m = N$ and $e_N \le e_{N-1}$\\

In this case, take $y_N\in Y_N\backslash Y_{N-1}$ to be the least deviation from $c_{N+1}y_{N+1}$ in $Y_N$,
$$
 \rho(c_{N+1}y_{N+1}-y_N, Y_{N-1} )\le \|c_{N+1}y_{N+1}-y_N\|=e_N\le e_{N-1}.
$$
Also since $\rho(y_N, Y_{N-1} )>0$, by the triangle inequality of $\rho(\cdot,Y)$, 
$$
 \rho(c_{N+1}y_{N+1}+\lambda y_N, Y_{N-1} )\ge  \lambda \rho(y_N, Y_{N-1} )- \rho(c_{N+1}y_{N+1}, Y_{N-1} )\xrightarrow[\lambda\to+\infty]{}+\infty.
$$
Therefore by the intermediate value theorem, there is $c_{N}\in\mathbb R$ such that
$$
\rho(c_{N+1}y_{N+1}+c_N y_N, Y_{N-1} )=e_{N-1}.
$$
\item $m < N$ and $e_N = e_{N-1}$\\

In this case, taking $c_N=1$, we directly get
$$
\rho(c_{N+1}y_{N+1} + c_Ny_N, Y_N ) = \rho(c_{N+1}y_{N+1} + y_N, Y_N )=e_N=e_{N-1}.
$$
\item $m < N$ and $e_N < e_{N-1}$\\

In this case, for any $y\in Y_N\backslash Y_{N-1}$, 
$$
\|c_{N+1}y_{N+1}-y\|>\rho(c_{N+1}y_{N+1},Y_N)=e_N.
$$
Let us denote $\|c_{N+1}y_{N+1}-y\|=e_N+\delta$ with $\delta>0$. Then by the triangle inequality,
$$
\|c_{N+1}y_{N+1}-cy\|\ge \|c_{N+1}y_{N+1}-y\|-|c-1|\|y\|=e_n+\delta-|c-1|\|y\|.
$$
For any $\epsilon\in(0,e_{N-1}-e_N)$, take $c\in\mathbb R$ such that $\delta-|c-1|\|y\|=\epsilon$ and $y_N=cy$.  
Then $y_N\in Y_N\backslash Y_{N-1}$ and
$$
\rho(c_{N+1}y_{N+1}-y_N,Y_{N-1})\le \|c_{N+1}y_{N+1}-y_N\|=e_N+\epsilon<e_{N-1}.
$$
Also since $\rho(y_N, Y_{N-1} )>0$, by the triangle inequality of $\rho(\cdot,Y)$, 
$$
 \rho(c_{N+1}y_{N+1}+\lambda y_N, Y_{N-1} )\ge  \lambda \rho(y_N, Y_{N-1} )- \rho(c_{N+1}y_{N+1}, Y_{N-1} )\xrightarrow[\lambda\to+\infty]{}+\infty.
$$
Therefore by the intermediate value theorem, there is $c_{N}\in\mathbb R$ such that
$$
\rho(c_{N+1}y_{N+1}+c_N y_N, Y_{N-1} )=e_{N-1}.
$$
\end{enumerate}
%%%%%%%%%%%%%%%%%
%%%%%%%%%%%%%%

In all these three cases, using  the reflexivity of the subspaces  and the properties  $\rho\,\,(\,\, .\,\, , Y)$,  in particular continuity  it can be shown that $$\rho(c_{N+1} y_{N+1}+c_Ny_N, Y_{N-1} )= e_{N-1}.$$ Continuing in this manner we can construct an element $$x =c_{N+1} y_{N+1}+\dots+c_1y_1$$ which has the desired properties.
\end{proof}
 Note that Nikol'skii's result is  restricted to a special sequence of the form
 $$d_0\geq d_1\geq \dots  \geq d_N \geq d_{N+1} =0 =d_{N+m}$$ for all $m\geq 1$.\\
 
  The general question is that, given a Banach space $X$ and a sequence of strictly increasing subspaces $\{Y_n\}$ of $X$  with a dense union and a decreasing null sequence $ \{d_n\}$, does there exist an $x\in X$ such that $\rho(x, Y_n) =d_n$? If such an element exists for all such subspaces $Y_n$ and numbers $d_n$ , then we say $X$ has the $(B)$-property. Besides Hilbert spaces \cite{Tyu2}, ( see \cite{AL} for an alternative proof)  there are no known examples of infinite-dimensional  spaces $X$ having the $(B)$-property. Nevertheless one can show the following relationship between the property (B) and the reflexivity.

\begin{theorem}
If $X$ satisfies the property (B) then X is reflexive.
\end{theorem}
\begin{proof}
Suppose   $d_1 = d_2=1$, for $n \geq 3$  the sequence $\{d_n\}$  is non-increasing and $\{d_n\} \to 0$.
Consider a pair of subspaces $Y_1\subset Y_2 \subset X$ such that
$$
Y_1= \{0\};\quad \quad Y_2= \ker (f)= \{ y\in X: \,\, f(y)=0\},
$$
where $f$ is a nonzero linear functional. Suppose the Bernstein's Lethargy Theorem holds.  In other words, there is an element $x\in X$ with prescribed best approximation errors. Since in this case

 $$ \rho(x, Y_1)= \rho(x, Y_2)= 1$$ and

$$\rho(x, Y_1)= \|x\|, \quad \rho(x, Y_2)= \displaystyle\frac{f(x)}{\|f\|}, $$  one has the equalities
$$ 1= \|x\|= \displaystyle\frac{f(x)}{\|f\|}.$$
 From the last equality we see that the functional $f$ assumes its supremum at $x$.   This implies that  $X$ is reflexive (see James' Theorem \cite{RCJ}).
\end{proof}

\begin{remark}
The next  step is to consider the existence of best approximation together with a sequence of nonnegative numbers $d_n$ satisfying $\displaystyle \lim_{n\to \infty} d_n=0$ and investigate the relationship between Bernstein Lethargy Theorem and reflexivity. 
\end{remark}
\noindent \textbf{Acknowledgement.} Authors would like to thank Jose Mar\'ia Almira for his valuable comments and to Yu. A. Skvortsov for his help in locating the reference \cite{Nik}.

\end{document}